\documentclass[a4,12pt]{amsart}

\usepackage{amssymb}
\usepackage{amstext}
\usepackage{amsmath}
\usepackage{amscd}
\usepackage{amsthm}
\usepackage{amsfonts}
\usepackage{cases}

\newenvironment{red}
{\relax\color{red}}
{\hspace*{.5ex}\relax}

\newcommand{\ber}{\begin{red}}
\newcommand{\er}{\end{red}}

\newenvironment{verd}
{\relax\color{magenta}}
{\hspace*{.5ex}\relax}

\newcommand{\bg}{\begin{verd}}
\newcommand{\eg}{\end{verd}}

\theoremstyle{defnition}

\newtheorem{Thm}{Theorem}
\newtheorem{Lem}[Thm]{Lemma}

\def\tp{{}^t\!}
\def\det{\mathrm{det}}
\def\perm{\mathrm{perm}}
\def\ep{\varepsilon}
\def\Mat{\mathrm{Mat}}
\def\stab{\mathrm{Stab}}
\def\rank{\mathrm{rank}}
\def\detab{\det^{(\alpha,\beta)}}
\def\detba{\det^{(\beta,\alpha)}}

\def\bA{\bar{A}}

\begin{document}
\title[Stabilizer group of generalized determinant]{Stabilizer group of generalized determinant}
\date{}

\author[Ryo Yamamoto]{Ryo Yamamoto}
\address{Department of Pure and Applied Mathematics, Graduate School of Information
Science and Technology, Osaka University, Suita, Osaka 565-0871, Japan}
\email{r-yamamoto@ist.osaka-u.ac.jp}

%\urladdr{}
%\thanks{2010 {\em Mathematics Subject Classification.} Primary: Secondary }
%\subjclass[2010]{}

\begin{abstract}
	In this paper, we introduce the notion of generalized determinant and determine the stabilizer group in $ GL(\Mat_n(K)) $ of the generalized determinant.
\end{abstract}
\maketitle

\section{Introduction.}

For an $ n \times n $ matrix $ A := (a_{ij})_{1 \le i, j \le n}$, we will define the \textit{even determinant} $ \det_{A_n}\!(A) := \sum_{\sigma \in A_n} \prod_{i=1}^{n}a_{i \sigma(i)} $ and the \textit{odd determinant} $ \det_{\bar{A}_n}\!(A) := \sum_{\sigma \in\bar{A}_n} \prod_{i=1}^{n}a_{i \sigma(i)} $ where $ \bar{A}_n $ is the set $ S_n \backslash A_n $.	
Let $ K $ be a field of characteristic 0.	
Let $ \alpha , \beta \in K $.
We introduce the \textit{generalized determinant} $\detab (A)$ of an $ n \times n $ matrix $A$ by 
$$\detab (A) := \alpha \, \det_{A_n}\!(A) + \beta \, \det_{\bar{A}_n}\!(A) .$$
Let $ X := (x_{ij})_{1 \le i, j \le n} $ where $ x_{ij} $ are the standard basis of $ \Mat_n^* $.
Write $ \det_n := \det(X) $, $ \perm_n := \perm(X)$ and $ \detab_n := \detab(X) $.
The stabilizer group $ \stab(f) $ of $ f \in \mathrm{Sym}^n(\Mat_n^*) $ is defined as $ \stab(f) := \{ \, T \in GL(\Mat_n(K)) \mid T \cdot f = f \, \}$.
If $ \alpha = -\beta \neq 0$, then $ \stab (\det^{(\alpha,-\alpha)}) = \stab (\det_n) $, which had been determined by Frobenius \cite{F} as follows.
\begin{Thm}[Frobenius]
	It holds that $\stab(\det_n)=\{\,X\mapsto PXQ$ or $P\tp XQ \mid \det P \, \det Q =1\,\} $.
	Here $ P , Q \in GL_n(K) $. 
\end{Thm}	
On the other hand, if $ \alpha = \beta \neq 0$, then $ \stab (\det^{(\alpha,\alpha)}) = \stab (\perm_n) $, which was determined by Marcus and May \cite{MM} as follows.		
\begin{Thm}[Marcus and May]
	Let $ n \ge 3 $.
	It holds that $\stab(\perm_n) = \{ \, X \mapsto LPXQR $ or $ LP \tp X QR  \mid  \det L \, \det R = 1 \, \}$ where $P $ and $ Q $ are permutation matrices, $L$ and $R$ are diagonal matrices.
\end{Thm}
Our purpose is to relate the two results by $ \mathbb{P}^1(K) $-family $ \{ \detab_n \mid [\alpha : \beta] \in \mathbb{P}^1(K) \} $.
Our main result is the following.
\begin{Thm}\label{main}
	Let $n \ge 5$ .
	If $ \alpha \neq \pm \beta $, then $\stab(\detab_n) = \stab(\det_n) \cap \stab(\perm_n)$.
\end{Thm}

\section{Notations.}
We define a submatrix $X_{l_1 \cdots l_r}^{k_1 \cdots k_r} $ of X by $X_{l_1 \cdots l_r}^{k_1 \cdots k_r} := (x_{k_il_j})_{1 \le i.j \le r} $.
Let $ P_r^{(\alpha,\beta)}(X) := (\detab (X_{l_1 \cdots l_r}^{k_1 \cdots k_r}))_{\substack{1 \le k_1 < \dots < k_r \le n \\ 1 \le l_1 < \dots < l_r \le n}} $ be a $\binom{n}{r} \times \binom{n}{r}$ matrix.
Specially, we write $ P_r(X) := P^{(1,0)}_r(X) $ and $\overline{P}_r(X) := P^{(0,1)}_r(X).$

\begin{Lem}\label{lin}
	Let $ n \ge 4 $.
	If $\detab(T(X))=\detab(X)$, then $ \detab(T(X)_{l_1 l_2}^{k_1 k_2})$ and $ \detba(T(X)_{l_1 l_2}^{k_1 k_2})$ are expressible as a linear combination of $ \detab(X_{l'_1 l'_2}^{k'_1 k'_2})$ and $ \detba(X_{l'_1 l'_2}^{k'_1 k'_2}) (1 \le k_1', k_2', l_1', l_2' \le n)$ respecticely.
\end{Lem}

\begin{proof}
	Let $Y:=T(X)$.
	We can write each entry of $X=T^{-1}(Y)$ as $x_{st} = \sum_{p,q=1}^{n} g_{st}^{pq} y_{pq}$.
	We first prove the case $ \detab(T(X)_{l_1 l_2}^{k_1 k_2})$.	
	Since $n \ge 4$, there exists $ \sigma \in A_n$ such that $ \sigma(k_1) = l_1$, $\sigma(k_2) = l_2$.	
	The permutation $ (l_1 ~ l_2)\sigma \in \bA_n $ satisfies $ (l_1 ~ l_2)\sigma(k_1) = l_2 $ and $ (l_1 ~ l_2)\sigma(k_2) = l_1$.
	We compute 
	\begin{eqnarray*}
		\detab(Y_{l_1 l_2}^{k_1 k_2}) &=& \detab\!\begin{pmatrix}
			y_{k_1 l_1} & y_{k_1 l_2} \\ 
			y_{k_2 l_1} & y_{k_2 l_2}
		\end{pmatrix} \\
		&=&\dfrac{\partial^{n-2}}{\partial y_{1 \sigma (1)} \cdots \widehat{\partial y_{k_1 \sigma (k_1)}} \cdots \widehat{\partial y_{k_2 \sigma (k_2)}}\cdots \partial y_{n \sigma (n)} }\detab(Y) \\
		&=&\dfrac{\partial^{n-2}}{\partial y_{1 \sigma (1)} \cdots \widehat{\partial y_{k_1 \sigma (k_1)}} \cdots \widehat{\partial y_{k_2 \sigma (k_2)}}\cdots \partial y_{n \sigma (n)} }\detab(X).
	\end{eqnarray*}

	To compute it, we use
	$$\frac{\partial}{\partial y_{p q}} \detab\!  (X_{l_1 \cdots l_r}^{k_1 \cdots k_r})  =  \sum_{s,t=1}^{n}  \frac{\partial x_{st}}{\partial y_{pq}} \frac{\partial}{\partial x_{st}} \detab\!  (X_{l_1 \cdots l_r}^{k_1 \cdots k_r}) $$
	and
	$$\frac{\partial}{\partial y_{p q}} \detba\!  (X_{l_1 \cdots l_r}^{k_1 \cdots k_r})  =  \sum_{s,t=1}^{n}  \frac{\partial x_{st}}{\partial y_{pq}} \frac{\partial}{\partial x_{st}} \detba\! (X_{l_1 \cdots l_r}^{k_1 \cdots k_r}).$$
	We have $ \frac{\partial x_{st}}{\partial y_{pq}} = g_{st}^{pq}\in K$ and $\frac{\partial}{\partial x_{st}} \detab(X_{l_1 \cdots l_r}^{k_1 \cdots k_r})$ is equal to $\detab(X_{l_1 \cdots \hat{t} \cdots l_r}^{k_1 \cdots \hat{s} \cdots k_r})$, $ \detba(X_{l_1 \cdots \hat{t} \cdots l_r}^{k_1 \cdots \hat{s} \cdots k_r})$ or 0.
	Hence, differentiating $n-2$ times, the lemma follows.
	
	We now turn to the case  $ \detba(T(X)_{l_1 l_2}^{k_1 k_2})$.
	Since $n \ge 4$, there exists $ \sigma \in A_n$ such that $ \sigma(k_1) = l_2, \, \sigma(k_2) = l_1$.
	The permutation $ (l_1 ~ l_2)\sigma \in \bA_n $ satisfies $ (l_1 ~ l_2)\sigma(k_1) = l_1, (l_1 ~ l_2)\sigma(k_2) = l_2$.
	Hence the same proof works for $ \detba(T(X)_{l_1 l_2}^{k_1 k_2})$.
\end{proof}

\begin{Lem}\label{rank1}
	Let $A  \in \Mat_n(K)$.
	If $P_2(A) = \overline{P}_2(A) = 0$, then $A$ is 0, a row matrix or a column matrix.
\end{Lem}
\begin{proof}
	Consider $A \ne 0$.
	Without loss of generality we can assume $ a_{11} \neq 0 $.	
	Then $ P_2 (A) = 0$ implies $a_{ij} = 0 \ (2 \le \forall i, \, j \le n)$.	
	If there is $ j \ge 2 $ such that $a_{1j} \neq 0$, then $a_{i1} = 0 \,(2 \le \forall i \le n)$ from $\overline{P}_2(A)=0$.	
	Thus $A$ is a row matrix.
	If there is $ i \ge 2 $ such that $a_{i1} \neq 0$, then $a_{1j} = 0 \,(2 \le \forall j \le n)$ from $\overline{P}_2(A)=0$.
	Thus $A$ is a column matrix.	
\end{proof}

\begin{Lem}\label{h1}
	Define $F_{ij} := T(E_{ij})$.
	If $ \alpha \neq \pm \beta $, then the number of non-zero entries in $F_{ij}$ is one.
\end{Lem}

\begin{proof}
	By Lemma \ref{lin}, both $\detab({F_{ij}}_{l_1 l_2}^{k_1 k_2})$ and $\detba({F_{ij}}_{l_1 l_2}^{k_1 k_2})$ are linear combination of  $ \detab({E_{ij}}_{l'_1 l'_2}^{k'_1 k'_2})$ and $ \detba({E_{ij}}_{l'_1 l'_2}^{k'_1 k'_2}) \ (1 \le k_1',\, k_2',\, l_1',\, l_2' \le n)$ respecticely.
	We thus get $ P_2^{(\alpha,\beta)}(F_{ij})= \alpha \, P_2(F_{ij}) + \beta \, \overline{P}_2(F_{ij})=0$ and $ P_2^{(\beta,\alpha)}(F_{ij})= \beta \, P_2(F_{ij}) + \alpha \, \overline{P}_2(F_{ij})=0$.
	Since $ \alpha^2 - \beta^2 \neq 0 $, $ P_2(F_{ij}) = \overline{P}_2(F_{ij}) = 0 $.
	Applying Lemma \ref{rank1}, we see that $ F_{ij} $ is a row matrix or a column matrix.
	
	Suppose that the number of non-zero entries in $F_{ij}$ is two or more.
	Let us assume that $F_{ij}$ is a row matrix with non-zero entries in the  $i'$th row.
	Since $F_{ij}+F_{it}=T(E_{ij}+E_{it}),\, F_{ij}+F_{tj}=T(E_{ij}+E_{tj})$, we have $P_2(F_{ij}+F_{it})=\overline{P}_2(F_{ij}+F_{it})=P_2(F_{ij}+F_{tj})=\overline{P}_2(F_{ij}+F_{tj})=0 \,(1 < \forall t \le n)$.
	By Lemma \ref{rank1}, $F_{ij}+F_{it}$ and $F_{ij}+F_{tj}$ are row matrices or column matrices, so that $ F_{it} $ and $ F_{tj} $ are row matrices lying in the $ i' $th row.
	However $\dim \mathrm{span}\{E_{i1},E_{i2},\dots,E_{in},E_{1j},\dots,E_{nj}\} > \dim \mathrm{span}\{E_{i'1},E_{i'2},\dots,E_{i'n}\}$, which contradicts the fact that $T$ is non-singular.
	The same proof works in the case that $F_{ij}$ is a column matrix.
\end{proof}

By Lemma \ref{h1}, we have $T(E_{ij})=c_{ij}E_{i'j'}$.
Since $ T $ is non-singular, $ c_{ij} \ne 0 $ and  $(i,j) \ne (s,t) $ implies $ (i',j') \ne (s',t')$.
Hereafter, we always assume $ \alpha \neq \pm \beta $ and define maps $ \mu,\lambda $ by $ T(E_{ij}) = c_{ij} E_{\mu(i,j)\lambda{(i,j)}} $.

\begin{Lem}\label{CPXQ}
	There exist permutation matrices $ P := (\delta_{i \sigma(j)})_{1 \le i,j \le n}$, $ Q := (\delta_{i \tau(j)})_{1 \le i,j \le n} $ where $\mathrm{sgn}(\sigma)\, \mathrm{sgn}(\tau)=1$, and a matrix $ C := (c_{ij})_{1 \le i,j \le n} \in M_{n,n}(K) $ with $ \forall c_{ij} \neq 0 $ such that $ T(X)=C*PXQ $ or $ T(X) = C*P \tp X Q $ (the operation $ * $ is the Hadamard product).
\end{Lem}

\begin{proof}
	We may assume that $\mu(1,1)=1$ and $\lambda(1,1)=1 $ by swapping rows or columns even number of times, that is, the number of row and column transpositions are both even or both odd.
	Since $ \rank(E_{11}+E_{22}) = 2 $, we have $ P_2(F_{11}+F_{22}) \neq 0$ or $ \overline{P}_2(F_{11}+F_{22}) \neq 0$. It follows that $ \mu(2,2) \ge 2 $ and $ \lambda(2,2) \ge 2$.
	Therefore, swapping rows or columns even number of times properly, we may assume that $\mu(2,2)=2$ and $\lambda(2,2)=2$.
	By continuing the same argument, we can assume $\mu(i,i)=i \, (1 \le \forall i \le n)$ and $\lambda(i,i)=i \, (1 \le \forall i \le n-2)$.
	There are two possibilities: (i) $\lambda(n-1,n-1)=n-1$ and $\lambda(n,n)=n$, (ii) $\lambda(n-1,n-1)=n$ and $\lambda(n,n)=n-1$.
	However, the case (ii) never happens because the coefficients of $ x_{11} \dots x_{n-1 n-1} x_{n n}$ in $ \detab (T(X))$ and in $ \detab (X)$ are different.
	Therefore, we conclude that
	$$ T(X) = C * P\begin{pmatrix}
	x_{11} &  &  & \mbox{\lower1.1ex\llap{\smash{\huge$*$}}} \\ 
	& \ddots &  &  \\ 
	&  & x_{n-1,n-1} &  \\ 
	\mbox{\strut\rlap{\smash{\huge$*$}}\quad} &  &  & x_{nn}
	\end{pmatrix} Q$$ where $\mathrm{sgn}(\sigma)\,\mathrm{sgn}(\tau)=1$.
	To continue the argument, we may assume $ P=Q=I_n $ that is $\mu(i,i)=i$ and $\lambda(i,i)=i \, (1 \le \forall i \le n)$ without loss of generality.
	
	By $P_2(E_{11}+E_{12})=\overline{P}_2(E_{11}+E_{12})=0$, we get $ \mu(1,2)=1$ or $ \lambda(1,2)=1$.
	By $P_2(E_{22}+E_{12})=\overline{P}_2(E_{22}+E_{12})=0$, we also get $ \mu(1,2)=2$ or $ \lambda(1,2)=2$.
	Combining these, we obtain two possibilities: (I) $\mu(1,2)=1 $ and $ \lambda(1,2)=2$, (I\hspace{-.1em}I) $\mu(1,2)=2 $ and $\lambda(1,2)=1$.
	
	Suppose first that (I) holds.
	Let $ 3 \le \gamma \le n$.
	By $P_2(E_{11}+E_{1\gamma})=\overline{P}_2(E_{11}+E_{1\gamma})=0$, we get $\mu(1,\gamma)=1$  or $\lambda(1,\gamma)=1$.
	By $P_2(E_{12}+E_{1\gamma})=\overline{P}_2(E_{12}+E_{1\gamma})=0$, we also get $\mu(1,\gamma)=1$ or $\lambda(1,\gamma)=2$.
	Combining these gives $\mu(1,\gamma)=1$.
	By $P_2(E_{\gamma \gamma}+E_{1\gamma})=\overline{P}_2(E_{\gamma \gamma}+E_{1\gamma})=0$, we obtain $\lambda(1,\gamma)=\gamma$.
	
	Let $ \delta \neq 1$.
	By $P_2(E_{11}+E_{\delta 1})=\overline{P}_2(E_{11}+E_{\delta 1})=0$, we have $\mu(\delta,1)=1$ or $ \lambda(\delta , 1)=1 $.　However $ \mu(1,\gamma) = 1 \, (1 \le \gamma \le n)$ gives $\mu(\delta,1) \neq 1$ as  $ T $ is non-singular.
	Hence $ \lambda(\delta , 1)=1 $.
	By $P_2(E_{\delta \delta}+E_{\delta 1})=\overline{P}_2(E_{\delta \delta}+E_{\delta 1})=0$, we obtain $ \mu(\delta,1)= \delta$.
	
	Let $ 1 < \gamma \neq \delta \le n$.
	By $P_2(E_{\delta 1}+E_{\delta \gamma})=\overline{P}_2(E_{\delta 1}+E_{\delta \gamma})=0$, we get $ \mu(\delta ,\gamma)= \delta$ or $ \lambda(\delta, \gamma) = 1 $ but the latter is impossible.
	By $P_2(E_{1 \gamma}+E_{\delta \gamma})=\overline{P}_2(E_{1 \gamma}+E_{\delta \gamma})=0$, we also get $ \lambda(\delta ,\gamma)= \gamma$.
	
	By the above argument, we obtain $ \mu(i,j) = i $ and $ \lambda(i,j) = j \, (1 \le \forall i ,\, j \le n)$ that is
	$$T(X)=C*PXQ \ (\mathrm{sgn}(\sigma) \, \mathrm{sgn}(\tau)=1).$$
	The same proof works for the case (I\hspace{-.1em}I) and we also obtain
	$$T(X)=C*P \tp XQ \ (\mathrm{sgn}(\sigma)\, \mathrm{sgn}(\tau)=1),$$
	if (I\hspace{-.1em}I) holds.
\end{proof}

\begin{Lem}\label{rackC}
	If $ n \ge 5 $, then $ \rank(C) = 1 $.
\end{Lem}	

\begin{proof}
	Comparing the coefficients of  $ \detab(X) $ and $ \detab(C * PXQ) $, we obtain
	\begin{numcases}
	{\hspace{3.0ex}}
	\label{se1}\hspace{-0.8ex} c_{1 w(1)}\dots c_{n w(n)}=1 \ (\forall w \in A_n) & $ \hspace{-1.2ex} (\beta = 0) $ \\
	\label{se2} \hspace{-0.8ex} c_{1 w(1)}\dots c_{n w(n)}=1 \ (\forall w \in \bA_n) & $ \hspace{-1.2ex} (\alpha = 0) $ \\
	\label{se3}\hspace{-0.8ex} c_{1 w(1)}\dots c_{n w(n)}=1 \ (\forall w \in S_n) & $ \hspace{-1.2ex} (\alpha \neq 0 , \beta \neq 0)  $.
	\end{numcases}
	Let us solve these simultaneous equations.
	
	For preperation, we first consider the case $ n=4 $.
	Let $ \beta = 0 $.
	By $ c_{11} c_{22} (c_{33} c_{44}) = 1 ,\, c_{12} c_{23} (c_{31} c_{44}) = 1 $ and $ c_{13} c_{21} (c_{32} c_{44}) = 1$, we can write\[ \begin{pmatrix}
	c_{11}       & c_{12}       & c_{13}       \\
	c_{21}       & c_{22}       & c_{23}       \\
	c_{31}c_{44} & c_{32}c_{44} & c_{33}c_{44}
	\end{pmatrix} = \begin{pmatrix}
	a_{11} & a_{12} & a_{13} \\ 
	a_{21} & u & v \\ 
	\frac{1}{a_{12} v}& \frac{1}{a_{13}a_{21}} & \frac{1}{a_{11}u}
	\end{pmatrix}  .\]
	Set $ a_{31} := c_{31} $.
	Then we have\[ c_{44} = \frac{1}{a_{12} a_{31} v},\, c_{32} = \frac{a_{12} a_{31} v}{a_{13}a_{21}},\, c_{33} = \frac{a_{12} a_{31} v}{a_{11} u}.\]
	Set $ a_{14} := c_{14} $.
	By $ c_{14}c_{23}c_{32}c_{41}=1 $, we have\[ a_{14} \cdot v \cdot \frac{a_{12} a_{31} v}{a_{13}a_{21}} \cdot c_{41}=1 \ \ \therefore c_{41} = \frac{a_{13} a_{21}}{a_{12}a_{14}a_{31}v^2}.\]
	By $ c_{12}c_{24}c_{33}c_{41}=1 $ and $ c_{13}c_{22}c_{34}c_{41}=1 $, we have\[ a_{12} \cdot c_{24} \cdot \frac{a_{12} a_{31} v}{a_{11}u} \cdot \frac{a_{13} a_{21}}{a_{12}a_{14}a_{31}v^2} = 1 \ \ \therefore c_{24} = \frac{a_{11} a_{14} u v}{a_{12}a_{13}a_{21}}\]
	\[ 
	a_{13} \cdot u \cdot c_{34} \cdot \frac{a_{13} a_{21}}{a_{12}a_{14}a_{31}v^2} = 1 \ \ \therefore c_{34} = \frac{a_{12} a_{14} a_{31} v^2}{a_{13}^2 a_{21} u}.\]
	By $ c_{13}c_{24}c_{31}c_{42} = 1 $ and $ c_{14}c_{21}c_{33}c_{42} = 1 $, we have\[ a_{13} \cdot \frac{a_{11} a_{14} u v}{a_{12}a_{13}a_{21}} \cdot a_{31} \cdot c_{42} = 1 \ \ \therefore c_{42} = \frac{a_{12} a_{21}}{a_{11} a_{14} a_{31} u v} \]
	\[ a_{14} \cdot a_{21} \cdot \frac{a_{12} a_{31} v}{a_{11} u} \cdot c_{42} = 1 \ \ \therefore c_{42} = \frac{a_{11} u}{a_{12} a_{14} a_{21} a_{31} v} .\]
	Combining these yields $$\displaystyle u= \pm \frac{a_{12} a_{21}}{a_{11}} .$$
	\noindent By $ c_{11}c_{24}c_{32}c_{43}=1 $ and $ c_{14}c_{22}c_{31}c_{43}=1 $, we have\[ a_{11} \cdot \frac{a_{11} a_{14} u v}{a_{12}a_{13}a_{21}} \cdot \frac{a_{12} a_{31} v}{a_{13}a_{21}} \cdot c_{43} =1 \therefore  c_{43} = \frac{a_{13}^2 a_{21}^2 }{a_{11}^2 a_{14}a_{31}u v^2}\]
	\[ 
	a_{14} \cdot u \cdot a_{31} \cdot c_{43} = 1 \ \ \therefore c_{43} = \frac{1}{a_{14} a_{31} u} .\]
	Combining these yields $$ v= \pm \frac{a_{13} a_{21}}{a_{11}} .$$
	Now, set $ a_{41} := c_{41} = \frac{a_{11}^2}{a_{12}a_{13}a_{14}a_{21}a_{31}}.$
	Summarizing the above, we obtain\begin{eqnarray} \label{C1}C= \begin{pmatrix}
	1 & 1 & 1 & 1 \\ 
	1 & \ep_u & \ep_v & \ep_u \ep_v \\ 
	1 & \ep_v & \ep_u \ep_v & \ep_u \\ 
	1 & \ep_u \ep_v & \ep_u & \ep_v
	\end{pmatrix} * \left( \frac{a_{i1}a_{1j}}{a_{11}} \right)_{1 \le i,j \le 4}  \end{eqnarray}
	where $ \ep_u,\ep_v \in \{+1,-1\} $ and $ a_{11}\cdots a_{14}a_{11}\cdots a_{41}=a_{11}^4 $.
	Conversely, the matrix C is a solution of the simultaneous equation \eqref{se1}.
	
	Let $ \alpha = 0 $.
	Interchanging the 3rd and the 4th row of \eqref{C1}, we obtain \begin{eqnarray} \label{C2}C= \begin{pmatrix}
	1 & 1 & 1 & 1 \\ 
	1 & \ep_u & \ep_v & \ep_u \ep_v \\ 
	1 & \ep_u \ep_v & \ep_u & \ep_v \\
	1 & \ep_v & \ep_u \ep_v & \ep_u 
	\end{pmatrix} * \left( \frac{a_{i1}a_{1j}}{a_{11}} \right)_{1 \le i,j \le 4}  \end{eqnarray}
	where $ \ep_u,\ep_v \in \{+1,-1\} $ and $ a_{11} \cdots a_{14}a_{11} \cdots a_{41}=a_{11}^4 $ as a solution of the simultaneous equation \eqref{se2}.
	
	Let $ \alpha \neq 0 $ and $ \beta \neq 0 $.
	Combining \eqref{C1} and \eqref{C2}, we obtain \begin{eqnarray} \label{C3}C= \left( \frac{a_{i1}a_{1j}}{a_{11}} \right)_{1 \le i,j \le 4}  \end{eqnarray}
	where $ a_{11}\cdots a_{14}a_{11}\cdots a_{41}=a_{11}^4 $ as a solution of the simultaneous equation \eqref{se3}.
	
	Let us consider the simultaneous equations for $ n \ge 5 $.
	We prove that the solution of the each simultaneous equations \eqref{se1},\eqref{se2},\eqref{se3} is expressible as
	\begin{eqnarray} \label{C}C= \left( \frac{a_{i1}a_{1j}}{a_{11}} \right)_{1 \le i,j \le n}  \end{eqnarray}	 where $ a_{11}\cdots a_{1n}a_{11}\cdots a_{n1}=a_{11}^n $ respectively by induction of the matrix size $ n $.
	
	In the case $ n = 5$ and $ \beta = 0$, using the fact that any solution of \eqref{se1} may be written as  \eqref{C1}, we can write
	\begin{eqnarray*}\begin{pmatrix}
			c_{11} & c_{12} & c_{13} & c_{14} \\ 
			c_{21} & c_{22} & c_{23} & c_{24} \\ 
			c_{31} & c_{32} & c_{33} & c_{34} \\ 
			c_{41}c_{55} & c_{42}c_{55} & c_{43}c_{55} & c_{44}c_{55}
		\end{pmatrix} =
		\renewcommand{\arraystretch}{2} \begin{pmatrix}
			\hspace{1ex} a_{11} \hspace{1ex} & a_{12}                             & a_{13}                                  & a_{14}                                  \\
			a_{21}                           & \ep_u\dfrac{a_{21}a_{12}}{a_{11}}  & \ep_v\dfrac{a_{21}a_{13}}{a_{11}}       & \ep_u \ep_v\dfrac{a_{21}a_{14}}{a_{11}} \\
			a_{31}                           & \ep_v\dfrac{a_{31}a_{12}}{a_{11}}  & \ep_u \ep_v\dfrac{a_{31}a_{13}}{a_{11}} & \ep_u\dfrac{a_{31}a_{14}}{a_{11}}       \\
			z                                & \ep_u \ep_v\dfrac{za_{12}}{a_{11}} & \ep_u\dfrac{za_{13}}{a_{11}}            & \ep_v\dfrac{za_{14}}{a_{11}}
		\end{pmatrix}
	\end{eqnarray*}
	where $ \ep_u,\ep_v \in \{+1,-1\} $ and $ a_{11} \cdots a_{14}a_{11} \cdots a_{31}z=a_{11}^4 $.
	Set $ a_{41} := c_{41} $.
	Then $ c_{55} = \frac{z}{a_{41}} ,\, c_{42}=\ep_u \ep_v \frac{a_{41}a_{12}}{a_{11}},\, c_{43}=\ep_u \frac{a_{41}a_{13}}{a_{11}},\, c_{44}=\ep_v \frac{a_{41}a_{14}}{a_{11}} $.
	Set $ a_{15} := c_{15}$.
	If $ \ep_u = -1$ or $\ep_v = -1$, there exist $ w \in A_n $ such that $ w(1) = 5 ,\, w(5)=1 ,\,  c_{i w(i)} = \frac{a_{i1}a_{1w(i)}}{a_{11}} \, (2 \le i \le 4) $ and $ w' \in A_n $ such that $ w'(1) = 5 ,\, w'(5)=1 ,\,  c_{i w'(i)} = -\frac{a_{i1}a_{1w'(i)}}{a_{11}} \, (2 \le i \le 4) $.
	Then  $ c_{15}c_{2w(2)}c_{3w(3)}c_{4w(4)}c_{51} \neq c_{15}c_{2w'(2)}c_{3w'(3)}c_{4w'(4)}c_{51} $ and one of the two cannot be equal to 1, a contradiction.
	Thus $ \ep_u = \ep_v = +1 $.
	The similar consideration applies to the case $ n=5 $ and $ \alpha=0 $.	
	
	Therefore, assuming \eqref{C} to hold for $ n-1 $, we can write 
	\begin{eqnarray*} \begin{pmatrix}
			c_{11} & \cdots & c_{1,n-1}  \\[0.5em]
			\vdots & \ddots & \vdots  \\[0.5em] 
			c_{n-2,1} & \cdots & c_{n-2,n-1}  \\
			c_{n-1,1}c_{nn} & \cdots & c_{n-1,n-1}c_{nn} 
		\end{pmatrix} = \renewcommand{\arraystretch}{2}\begin{pmatrix}
			a_{11} & \cdots & \cdots & a_{1,n-1} \\ 
			\vdots & \dfrac{a_{21}a_{12}}{a_{11}} & \dots & \dfrac{a_{21}a_{1,n-1}}{a_{11}} \\ 
			\vdots & \vdots & \ddots & \vdots \\[-0.4em]
			a_{n-2,1} & \dfrac{a_{n-2,1}a_{12}}{a_{11}} & \dots & \dfrac{a_{n-2,1}a_{1,n-1}}{a_{11}} \\ 
			z & \dfrac{za_{12}}{a_{11}} & \dots & \dfrac{za_{1,n-1}}{a_{11}}
		\end{pmatrix}
	\end{eqnarray*}
	where $ a_{11}\cdots a_{1,n-1}a_{11}\cdots a_{n-2,1}z=a_{11}^{n-1} $.
	Set $ a_{n-1,1} := c_{n-1,1} $.
	Then we have \[ c_{nn} = \dfrac{z}{a_{n-1,1}}, ~ c_{n-1,j}=\dfrac{a_{n-1,1}a_{1j}}{a_{11}} \, (1 \le j \le n-1) .\]
	Set $ a_{1n} := c_{1n} $.
	By $ c_{1n}(c_{2 w (2)} \cdots c_{n-1 w(n-1)})c_{n1} = 1$ for $ w \in A_n $ such that $ w (1) = n , \,w (n) = 1 $,
	\[ a_{1n} \cdot \dfrac{a_{12}\cdots a_{1,n-1}a_{21}\cdots a_{n-1,1}}{a_{11}^{n-2}} \cdot c_{n1} = 1 \]
	\[ \therefore c_{n1} = \dfrac{a_{11}^{n-2}}{a_{12}\cdots a_{1n}a_{21}\cdots a_{n-1,1}} = \dfrac{a_{11}z}{a_{1n}a_{n-1,1}}.
	\]
	Set $ a_{n1} := c_{n1} $.
	Then
	\[ z = \dfrac{a_{1n} a_{n-1,1}a_{n1}}{a_{11}}.
	\]
	It follows that
	$ c_{ij} = \frac{a_{i1}a_{1j}}{a_{11}} $ for $1 \le i ,\, j \le n-1 $ and $(i, j) = (1, n)$, $(n, 1)$, $(n, n)$ where $ a_{11} \cdots a_{1n} a_{11} \cdots a_{n1} = a_{11}^n$.
	By $ c_{in} (c_{1w(1)} \cdots \widehat{c_{in}} \cdots c_{n-1w(n-1)} )c_{n1} = 1 $, we have
	\[ c_{in} \cdot \dfrac{a_{12}\cdots a_{1,n-1}a_{21}\cdots a_{n-1,1}}{a_{11}^{n-3} a_{i1}} \cdot a_{n1} = 1 \]
	\[ \therefore c_{in} = \dfrac{a_{11}^{n-3} a_{i1}}{a_{12}\cdots a_{1,n-1}a_{21}\cdots a_{n1}} = \dfrac{a_{i1}a_{1n}}{a_{11}} .\]
	By $ (c_{1w(1)} \cdots c_{n-1w(n-1)}) c_{nj}= 1 $, we have
	\[ \dfrac{a_{12} \cdots a_{1n} a_{11} \cdots a_{n-1,1}}{a_{11}^{n-1} a_{1j} } \cdot c_{nj} = 1 \]
	\[ \therefore c_{nj} = \dfrac{a_{11}^{n-1} a_{1j} }{a_{12} \cdots a_{1n} a_{11} \cdots a_{n-1,1}} = \dfrac{a_{n1} a_{1j}}{a_{11}}.\]
	From the above, we obtain \eqref{C}.
	
	Each of $ 2 \times 2 $ minor determinants of C is 0, so that $ \rank(C) = 1 $ follows.	
\end{proof}

\begin{proof}[Proof of Theorem \ref{main}]
	Let $ T^{-1} \in \stab (\detab_n) $ and $ \alpha \neq \pm \beta $.
	By Lemma \ref{h1} and \ref{CPXQ}, We can write $ T(X)=C*PXQ $ or $ T(X) = C*P \tp X Q $ where $ \text{sgn}(\sigma)\, \text{sgn}(\tau) = 1 $.
	By Lemma \ref{rackC}, we can write $ c_{ij} = l_ir_j \,(1 \le i,\,j \le n)$ where $ l_1 \cdots l_n r_1 \cdots r_n =1$.
	Set $ L := \text{diag}(l_1,\dots, l_n) $ and $ R := \text{diag}(r_1,\dots, r_n) $.
	Then we obtain  $ T(X)=LPXQR $ or $ T(X) = LP \tp X QR $ where $ \text{sgn}(\sigma)\,\text{sgn}(\tau) = 1 $, which proves the theorem.	
\end{proof}

%%%%%%%%%%%%%%%%%%%%%%%%%%%%%%%%%%%%%%%%%%%%%%%%%%%%%%%%%%

\bibliographystyle{amsalpha}

\end{document}